\documentclass[a4paper]{amsart}

\usepackage[T1]{fontenc}
\usepackage{amsmath, amssymb, amsthm, mathrsfs, mathtools, graphicx}
\usepackage[backref=page]{hyperref}
\usepackage{enumitem, xspace, ifthen, comment}
\usepackage[all]{xy}
\usepackage[nameinlink, capitalize]{cleveref}
\usepackage[dvipsnames]{xcolor}
\usepackage{tikz}

\setlist[enumerate]{
  label=(\thethm.\arabic*),
  before={\setcounter{enumi}{\value{equation}}},
  after={\setcounter{equation}{\value{enumi}}},
  itemsep=1ex
}

\setlist[itemize]{
  leftmargin=*,
  topsep=1ex,
  itemsep=1ex,
  label=$\circ$
}

\newcommand{\mypagesize}{
  \addtolength{\textwidth}{12pt}
  \addtolength{\textheight}{42pt}
  \calclayout
}
\mypagesize

\newtheorem*{thm-plain}{Theorem}
\newtheorem{thm}{Theorem}[section]
\newtheorem{lem}[thm]{Lemma}
\newtheorem{prp}[thm]{Proposition}
\newtheorem{cor}[thm]{Corollary}

\newtheorem{ques}[thm]{Question}
\newtheorem{cons}[thm]{Construction}
\numberwithin{equation}{thm}

\theoremstyle{definition}
\newtheorem{dfn}[thm]{Definition}

\newtheorem*{dfn-plain}{Definition}

\theoremstyle{remark}

\newtheorem*{rem-plain}{Remark}

\newcommand{\inv}{^{-1}}
\newcommand{\from}{\colon}

\newcommand{\lto}{\longrightarrow}

\newcommand{\x}{\times}
\newcommand{\inj}{\hookrightarrow}

\newcommand{\bij}{\overset\sim\lto}
\newcommand{\isom}{\cong}
\newcommand{\defn}{\coloneqq}

\newcommand{\tensor}{\otimes}
\newcommand{\id}{\mathrm{id}}

\newcommand{\wt}{\widetilde}

\newcommand{\dual}{^{\smash{\scalebox{.7}[1.4]{\rotatebox{90}{\textup\guilsinglleft}}}}}
\newcommand{\ddual}{^{\smash{\scalebox{.7}[1.4]{\rotatebox{90}{\textup\guilsinglleft} \hspace{-.5em} \rotatebox{90}{\textup\guilsinglleft}}}}}

\newcommand{\factor}[2]{\left. \raise 2pt\hbox{$#1$} \right/\hskip -2pt \raise -2pt\hbox{$#2$}}

\newdir{ ir}{{}*!/-5pt/@^{(}} % Injective arrow, for \xymatrix{}.
\newdir{ il}{{}*!/-5pt/@_{(}} % Usage: \ar@{ ir->}, \ar@{ il->} }}

\def\rd#1.{\lfloor{#1}\rfloor}
\def\rp#1.{\lceil{#1}\rceil}
\def\tw#1.{\langle{#1}\rangle}

\renewcommand{\O}[1]{\mathscr{O}_{#1}}

\newcommand{\Reg}[1]{{#1}_{\mathrm{reg}}}

\newcommand{\cc}[2]{\mathrm{c}_{#1}(#2)}

\def\Hnought#1.#2.{\mathit{\Gamma} \!\left( #1, #2 \right)}
\def\HH#1.#2.#3.{\mathrm{H}^{#1} \!\left( #2, #3 \right)}
\def\hh#1.#2.#3.{h^{#1} \!\left( #2, #3 \right)}
\def\RR#1.#2.#3.{R^{#1} #2_* #3}
\def\HHc#1.#2.#3.{\mathrm{H}_{\mathrm{c}}^{#1} \!\left( #2, #3 \right)}
\def\Hh#1.#2.#3.{\mathrm{H}_{#1} \!\left( #2, #3 \right)}
\def\Hom#1.#2.{\mathrm{Hom} \!\left( #1, #2 \right)}
\def\End#1.{\mathrm{End} \!\left( #1 \right)}
\def\sHom#1.#2.{\mathscr{H}\!om \!\left( #1, #2 \right)}
\def\Ext#1.#2.#3.{\mathrm{Ext}^{#1} \!\left( #2, #3 \right)}
\def\sExt#1.#2.#3.{\mathscr{E}\!xt^{#1} \!\left( #2, #3 \right)}

\renewcommand{\P}[1]{\mathbb P^{#1}}

\newcommand{\kahler}{K{\"{a}}hler\xspace}
\newcommand{\mfs}{Mori fibre space\xspace}

\DeclareMathOperator{\Pic}{Pic}
\DeclareMathOperator{\Picn}{Pic^\circ}
\DeclareMathOperator{\Gal}{Gal}

\DeclareMathOperator{\Def}{Def}

\renewcommand{\theta}{\vartheta}
\renewcommand{\phi}{\varphi}

\newcommand{\Z}{\ensuremath{\mathbb Z}}
\newcommand{\Q}{\ensuremath{\mathbb Q}}

\newcommand{\C}{\ensuremath{\mathbb C}}

\newcommand{\bP}{\ensuremath{\mathbb P}}

\newcommand{\frS}{\mathfrak S} \newcommand{\frX}{\mathfrak X}

 \newcommand{\sE}{\mathscr E} 
  
  \newcommand{\sL}{\mathscr L}
\newcommand{\sM}{\mathscr M}  \newcommand{\sO}{\mathscr O}
\newcommand{\sP}{\mathscr P} \newcommand{\sQ}{\mathscr Q} 
  
  \newcommand{\sX}{\mathscr X}
\newcommand{\sY}{\mathscr Y} 

  \newcommand{\cC}{\mathcal C}

  \newcommand{\cR}{\mathcal R}

\newcommand{\flaccid}{scenic\xspace}
\newcommand{\Flaccid}{Scenic\xspace}

\definecolor{forrest}{RGB}{81,133,49}
\definecolor{mydarkblue}{RGB}{10,92,153}

\title{On the Kodaira problem for uniruled K{\"{A}}hler spaces}

\author{Patrick Graf}
\address{Department of Mathematics, University of Utah, 155~South 1400~East, Salt Lake City, UT 84112}
\email{\href{mailto:patrick.graf@uni-bayreuth.de}{patrick.graf@uni-bayreuth.de}}
\urladdr{\href{http://www.pgraf.uni-bayreuth.de/en/}{www.graficland.uni-bayreuth.de}}

\author{Martin Schwald}
\address{Fakult\"at f\"ur Mathematik, Universit\"at Duisburg--Essen, 45117 Essen, Germany}
\email{\href{mailto:martin.schwald@uni-due.de}{martin.schwald@uni-due.de}}
\urladdr{\href{http://www.esaga.uni-due.de/martin.schwald/}{www.esaga.uni-due.de/martin.schwald/}}

\date{May 6, 2020}
\thanks{The first author was partially supported by a DFG Research Fellowship.
The second author was partially supported by the DFG Collaborative Research Centre SFB/TR 45.}

\hypersetup{
  pdfauthor={Patrick Graf, Martin Schwald},
  pdftitle={On the Kodaira problem for uniruled Kahler spaces},
  pdfkeywords={Kahler manifolds, uniruled Kahler spaces, Mori fibrations, algebraic approximation, small projective deformations, locally trivial deformations},
  pdfstartview={Fit},
  pdfpagelayout={OneColumn},
  pdfpagemode={UseNone},
  linktoc=all,
  breaklinks,
  linkcolor=[RGB]{0 0 96},
  citecolor=[RGB]{96 0 0},
  urlcolor=[RGB]{0 96 0},
  colorlinks}

\begin{document}

\begin{abstract}
We discuss the Kodaira problem for uniruled \kahler spaces.
Building on a construction due to Voisin, we give an example of a uniruled \kahler space~$X$ such that every run of the $K_X$-MMP immediately terminates with a \mfs, yet~$X$ does not admit an algebraic approximation.
Our example also shows that for a Mori fibration, approximability of the base does not imply approximability of the total space.
\end{abstract}

\maketitle

\thispagestyle{empty}

\section{Introduction}

A fundamental problem in complex algebraic and \kahler geometry is to determine the relationship between smooth projective varieties and compact \kahler manifolds.
Since a compact complex manifold is projective if and only if it admits a \kahler form whose cohomology class is rational, the following question suggests itself.

\begin{ques}[Kodaira problem] \label{Kodaira problem}
Is it possible to make any compact \kahler manifold~$X$ projective by an arbitrarily small deformation $X_t$ of its complex structure?
\end{ques}

\noindent
Such a deformation will be called an \emph{algebraic approximation} of $X$.
See \cref{def alg approx} for the precise notion.

Kodaira proved that every compact \kahler surface can be deformed to an algebraic surface~\cite[Thm.~16.1]{Kod63}.
In higher dimensions, the Kodaira problem remained open until in~\cite{Voi04} Voisin gave counterexamples (of Kodaira dimension $\kappa = 0$) in any dimension $\ge 4$.
In~\cite{Voi06}, she even constructed examples (uniruled and of even dimension $\ge 10$) of compact \kahler manifolds $X_{\mathrm{Voi}}$ such that no compact complex manifold $X'$ bimeromorphic to $X_{\mathrm{Voi}}$ admits an algebraic approximation.

At first sight, this seems to provide a definite negative answer to the Kodaira problem.
However, from the viewpoint of the Minimal Model Program (MMP), it is natural to take into account also singular bimeromorphic models.
A most influential statement in this direction is Peternell's conjecture that minimal models of compact \kahler manifolds should admit an algebraic approximation.
This conjecture has recently spawned substantial progress on the Kodaira problem in dimension three~\cite{AlgApprox, ClaudonHoeringKahlerGroups, Lin16, Lin17a, Lin17b}.

That said, an obvious desire arises to revisit Voisin's example $X_{\mathrm{Voi}}$ and to investigate whether some singular model of it is approximable.
By construction, $X_{\mathrm{Voi}}$ comes equipped with a bimeromorphic map to a mildly singular \kahler space $X$, and the map $X_{\mathrm{Voi}} \to X$ is a (composition of) $K_X$-negative extremal contractions.
Our first result shows that this new space $X$ does not admit an algebraic approximation.
Of course, one would then like to contract (or flip) further extremal rays, hoping to arrive at an approximable model.
We show that this is impossible: $X$ is minimal in the sense that every run of the $K_X$-MMP immediately yields a \mfs.
Actually, we prove an even stronger statement---see~\labelcref{main.cont} below.

In a slightly different direction, one might consider the Mori fibrations of a given uniruled space and ask whether approximability of the base of such a fibration implies approximability of the total space.
Our example $X$ shows that this is likewise not the case.
Summing up, what we prove is the following:

\begin{thm}[Non-approximable minimal uniruled \kahler space] \label{main}
For every even number $n \ge 10$, there exists an $n$-dimensional uniruled compact \kahler space $X$ with the following properties:
\begin{enumerate}
\item \label{main.sg} $X$ is simply connected and has only terminal quotient singularities.
\item \label{main.cont} Any bimeromorphic map $X \to X'$ to a normal complex space $X'$ is an isomorphism.
In particular, every run of the $K_X$-MMP immediately terminates with a Mori fibration.
\item \label{main.mori} There is a Mori fibration $X \to Y$ such that $Y$ admits an algebraic approximation.
\item \label{main.approx} $X$ does not admit an algebraic approximation.
\end{enumerate}
\end{thm}

In~\cite{Lin17b}, Lin has shown that any uniruled \kahler \emph{threefold} is approximable.
Our result shows that in higher dimensions, the situation becomes considerably more complicated.
We are not aware of any natural condition on a uniruled \kahler space that would guarantee, at least conjecturally, the existence of an algebraic approximation.
This suggests that higher-dimensional uniruled spaces are quite pathological from this point of view.

\subsection*{Open questions}

We cannot exclude the possibility that our example $X$ is bimeromorphic to an approximable \kahler space $X'$ in some haphazard way.
But by~\labelcref{main.cont}, the existence of such an $X'$ would not be explained by general principles such as the MMP.
Hence from a systematic viewpoint, we do not expect such an $X'$ to exist.

Nevertheless, this is of course an interesting question.
All we can say at the moment is that such an $X'$ would necessarily have non-rigid singularities.
This follows from our proof of~\labelcref{main.approx}.

\subsection*{Acknowledgements}

This project was started during a stay at the Ma\-the\-ma\-tisch\-es Forschungs\-institut Ober\-wolfach, whose hospitality is unmatched.
We also thank the referee for his efforts to improve the paper.

\section{Basic facts and definitions}

\subsection*{Complex spaces}

All complex spaces are assumed to be separated, connected and reduced, unless otherwise stated.

An irreducible compact complex space $X$ is said to be \emph{of Fujiki class $\cC$} (or \emph{in $\cC$}, for short) if it is bimeromorphic to a compact \kahler manifold.
We say that $X$ is \emph{Moishezon} if its field of meromorphic functions $\sM(X)$ has maximal transcendence degree $\operatorname{trdeg}_\C \sM(X) = \dim X$.
Being Moishezon is equivalent to being bimeromorphic to a projective manifold.
We say that a (not necessarily irreducible) compact complex space is Moishezon if each of its irreducible components is Moishezon.

\subsection*{Resolution of singularities}

A \emph{resolution of singularities} of a complex space $X$ is a proper bimeromorphic morphism $f \from \wt X \to X$, where $\wt X$ is smooth.
We say that the resolution is \emph{projective} if $f$ is a projective morphism.
In this case, if $X$ is projective (resp.~compact \kahler) then so is $\wt X$.
A resolution is said to be \emph{strong} if it is an isomor\-phism over the smooth locus of $X$.
A \emph{log resolution} is a resolution whose exceptional locus is a simple normal crossings divisor in $\wt X$.
It will be important for us that resolving singularities is not only possible any-old-how, but there is a canonical way of doing so:

\begin{thm}[Functorial resolutions] \label{funct res}
There exists a \emph{resolution functor} which assigns to any complex space $X$ a strong projective log resolution $\pi_X \from \cR(X) \to X$, such that $\cR$ commutes with smooth maps in the following sense:
For any smooth morphism $f \from W \to X$, there is a unique smooth morphism $\cR(f) \from \cR(W) \to \cR(X)$ such that the following diagram is a fibre product square.
\[ \xymatrix{
\cR(W) \ar^-{\cR(f)}[rr] \ar_-{\pi_W}[d] & & \cR(X) \ar^-{\pi_X}[d] \\
W \ar^-f[rr] & & X.
} \]
\end{thm}

\begin{proof}
See~\cite[Thm.~3.45]{Kol07}.
\end{proof}

\subsection*{\kahler spaces} \label{sec kahler spaces}

While we will not work directly with the definition of a singular \kahler space, we include the definition here for the reader's convenience.

\begin{dfn}[\kahler space] \label{def kahler}
Let $X$ be a normal complex space.
A \emph{\kahler form} $\omega$ on $X$ is a \kahler form $\omega^\circ$ on the smooth locus $\Reg X \subset X$ such that $X$ can be covered by open sets $U_\alpha$ with the following property: there is an embedding $U_\alpha \inj W_\alpha$ of $U_\alpha$ as an analytic subset of an open set $W_\alpha \subset \C^{n_\alpha}$ and a \kahler form $\wt\omega_\alpha$ on $W_\alpha$ such that
\[ \omega^\circ \big|_{U_\alpha \cap \Reg X} = \wt\omega_\alpha \big|_{U_\alpha \cap \Reg X}. \]
A normal complex space $X$ is said to be \emph{\kahler} if there exists a \kahler form on $X$.
\end{dfn}

For example, the analytification of a normal complex projective variety is a \kahler space.

\subsection*{Deformation theory} \label{sec def theory}

We collect some notation and basic facts from deformation theory.

\begin{dfn}[Deformations of complex spaces]
A \emph{deformation} of a complex space $X$ is a flat morphism $\frX \to (S, 0)$ from a (not necessarily reduced) complex space $\frX$ to a complex space germ $(S, 0)$, together with the choice of an isomorphism $\frX_0 \isom X$, where we write $\frX_s \defn \pi\inv(s)$ for the fibre over any $s\in S$.
We usually suppress both the base point $0 \in S$ and the choice of the isomorphism from notation.
Deformations of complex space \emph{germs} are defined similarly.
\end{dfn}

\begin{dfn}[Locally trivial deformations]
A deformation $\pi \from \frX \to S$ is called \emph{locally trivial} if for every $x \in \frX_0$ there exist open subsets $0 \in S^\circ \subset S$ and $x \in U \subset \pi\inv(S^\circ)$ and an isomorphism
\[ \xymatrix{
U \ar^-\sim[rr] \ar_-\pi[dr] & & (\frX_0 \cap U) \x S^\circ \ar^-{\operatorname{pr}_2}[dl] \\
& S^\circ. &
} \]
\end{dfn}

\begin{dfn}[Rigid singularities] \label{def rigid sing}
A complex space germ $(X, x)$ is called \emph{rigid} if every deformation of $X$ is trivial.
A complex space $X$ is said to have \emph{rigid singularities} if for each $x \in X$, the germ $(X, x)$ is rigid.
Equivalently, every deformation of $X$ is locally trivial.
\end{dfn}

\begin{dfn}[Algebraic approximations] \label{def alg approx}
Let $X$ be a compact complex space and $\pi \from \frX \to S$ a deformation of $X$.
Consider the set of projective fibres
\[ S^{\mathrm{alg}} \defn \big\{ s \in S \;\big|\; \frX_s \text{ is projective} \big\} \subset S \]
and its closure $\overline{S^{\mathrm{alg}}} \subset S$.
We say that $\frX \to S$ is an \emph{algebraic approximation of $X$} if $0 \in \overline{S^{\mathrm{alg}}}$.
\end{dfn}

\section{Voisin's example: construction and properties}

The aim of this section is threefold.
First we recall Voisin's example from~\cite{Voi06} in order to fix notation and for the reader's convenience.
Second, we investigate some of its properties which have not been discussed by Voisin.
In particular, we take a closer look at the singularities arising in the construction.
Third, we make the example as concrete as possible by providing an explicit example of Voisin's ``property (\textasteriskcentered)''.

\begin{dfn}[\Flaccid tori]
A \emph{\flaccid torus} is a pair $(T, \phi)$ consisting of an $n$-dimensional complex torus $T$ and an endomorphism $\phi \from T \to T$ such that the induced map $\phi_* \from \Hh1.T.\C. \to \Hh1.T.\C.$ has the following property:
the eigenvalues $\mu_1, \dots, \mu_{2n}$ of $\phi_*$ are pairwise distinct, none of them are real, and the Galois group $\Gal \left( \factor{ \Q( \mu_1, \dots, \mu_{2n} ) }{ \Q } \right)$ is the full symmetric group $\frS_{2n}$.
\end{dfn}

\subsection{Polynomials with large Galois group}

In~\cite[\S1]{Voi04}, it is explained how to construct a \flaccid torus starting from a rank $2n$ lattice $\Gamma$ and an endomorphism $\phi_\Z$ of $\Gamma$ whose characteristic polynomial has full symmetric Galois group and no real roots, as above.
So for us it only remains to give an example of such a lattice and endomorphism.
We will see that such examples are abundant for any value of $n$.
The following theorem gives a criterion for the characteristic polynomial $f$ of $\phi_\Z$ to have the desired Galois group.

\begin{thm}[Polynomials with full symmetric Galois group] \label{vanderWaerden}
Let $f \in \Z[x]$ be a monic polynomial of degree $d$ with the following properties:
\begin{enumerate}
\item The image of $f$ in $\mathbb F_2[x]$ is irreducible.
\item The image of $f$ in $\mathbb F_3[x]$ splits into a linear factor and an irreducible factor of degree $d - 1$.
\item The image of $f$ in $\mathbb F_5[x]$ splits into an irreducible quadratic factor and one or two irreducible factors of odd degree.
\end{enumerate}
Then the splitting field $K$ of $f$ has Galois group $\Gal(K/\Q) = \frS_d$. \qed
\end{thm}

\begin{proof}
See~\cite[\S66, p.~204]{vanderWaerdenAlgebra}.
\end{proof}

For any prime $p$, there exist irreducible polynomials over $\mathbb F_p$ of any given degree.
Thus for any $d$ we can find monic polynomials $f_2, f_3, f_5 \in \Z[x]$ which over $\mathbb F_2, \mathbb F_3, \mathbb F_5$ split as described in \cref{vanderWaerden}.
Then $f \defn -15 f_2 + 10 f_3 + 6 f_5 + 30 k$ is, for any $k \in\Z$, a monic polynomial of degree $d$ with Galois group $\frS_d$.
If $d = 2n$ is even and $k \gg 0$ sufficiently big, then this polynomial does not have any real roots.
For a concrete example, consider the case $n = 4$, which is the smallest value to which~\cite{Voi06} applies.
Then we may take
\begin{align*}
f & = -15 \underbrace{ (x^8 + x^4 + x^3 + x + 1) }_{\text{irreducible mod $2$}} + 10 (x - 1) \underbrace{ (x^7 + x^2 + 2) }_{\text{irred.~mod $3$}} \\[2ex]
  & \hspace{2em} + 6 \underbrace{(x^2 + 2) (x^3 + x + 1) (x^3 + x + 4)}_{\text{each factor irreducible mod $5$}} + 120 \\[2ex]
  & = x^8 - 10 x^7 + 24 x^6 + 30 x^5 + 15 x^4 + 85 x^3 + 26 x^2 + 65 x + 133 \in \Z[x].
\end{align*}
For any $f$ as above, set $\Gamma \defn \factor{\Z[x]}{(f)}$ and take $\phi_\Z \from \Gamma \to \Gamma$ to be multiplication by $x$.
Since $f$ is monic, $\Gamma$ is a lattice and by construction, the minimal polynomial of $\phi_\Z$ is $f$.
By degree reasons, $f$ is then also the characteristic polynomial of $\phi_\Z$.

\subsection{Voisin's construction} \label{voisin construction}

Before we sum up the construction in~\cite{Voi06}, recall the following standard definitions.

\begin{dfn}[Dual torus, Poincar\'e bundle, Kummer construction]
Let $T$ be an $n$-dimensional complex torus.
\begin{enumerate}
\item \label{dual torus} The \emph{dual torus} of $T$ is defined as
\[ T\dual \defn \factor{\HH1.T.\O T.}{\HH1.T.\Z.}. \]
By the exponential sequence on $T$, the map $\exp \from \HH1.T.\O T. \to \HH1.T.\O T^*.$ induces an isomorphism
\[ T\dual \bij \Picn(T) \defn \ker \left( \HH1.T.\O T^*. \xrightarrow{\; \mathrm c_1 \;} \HH2.T.\Z. \right). \]
This identifies $T\dual$ with $\Picn(T)$, the group of topologically trivial holomorphic line bundles on $T$.
For a point $t \in T\dual$, we will denote the corresponding line bundle by $\sL_t$.
\item \label{poincare} The \emph{Poincar\'e bundle} $\sP$ on $T \x T\dual$ is the line bundle, unique up to isomorphism, with the following two properties:
\begin{itemize}
\item For all $t \in T\dual$, we have $\sP \big|_{ T \x \{ t \} } \cong \sL_t$.
\item $\sP \big|_{ \{ 0 \} \x T\dual } \isom \O{T\dual}$ is trivial.
\end{itemize}
If $\phi$ is an endomorphism of $T$, we define the \emph{twisted Poincar\'e bundle} on $T \x T\dual$ as $\sP_\phi \defn (\phi, \id_{T\dual})^* \sP$.
In particular, we have $\sP = \sP_{\id_T}$.
\item Consider the automorphism $i$ of $T$ given by $t \mapsto -t$.
Its fixed points are exactly the $2^{2n}$ two-torsion points of $T$, the set of which we denote by $\tau_2(T)$.
The \emph{(singular) Kummer variety} associated to $T$ is
\[ K(T) \defn \factor{T}{\langle i \rangle}. \]
\end{enumerate}
\end{dfn}

\begin{lem}[Pulling back the Poincar\'e bundle] \label{poincare pullback}
Let $T$ be a complex torus with an endomorphism $\phi$. We have the following isomorphisms:
\begin{align*}
(-\id_T, \id_{T\dual})^* \sP_\phi & \isom \sP_\phi\inv, \\
(\id_T, -\id_{T\dual})^* \sP_\phi & \isom \sP_\phi\inv.
\end{align*}
These isomorphisms are unique if we require them to respect a choice of trivialization $\sP \big|_{(0,0)} \isom \C$ fixed in advance.
\end{lem}

\begin{proof}
The involution $-\id_T$ acts as $-\id$ on $\pi_1(T)$ and hence also on $\Picn(T)$.
Therefore, for all $t \in T\dual$ we have
\begin{align*}
(-\id_T, \id_{T\dual})^* \sP \big|_{T \x \{ t \}} & \isom (-\id_T)^* \sL_t \isom \sL_t\inv \qquad \text{and} \\
(\id_T, -\id_{T\dual})^* \sP \big|_{T \x \{ t \}} & \isom \sP \big|_{T \x \{ -t \}} \isom \sL_{-t} \isom \sL_t\inv,
\end{align*}
as well as
\begin{align*}
(-\id_T, \id_{T\dual})^* \sP \big|_{\{ 0 \} \x T\dual} & \isom \id_{T\dual}^* \O{T\dual} \isom \O{T\dual} \qquad \text{and} \\
(\id_T, -\id_{T\dual})^* \sP \big|_{\{ 0 \} \x T\dual} & \isom (-\id_{T\dual})^* \O{T\dual} \isom \O{T\dual}.
\end{align*}
This shows that $(-\id_T, \id_{T\dual})^* \sP\inv$ and $(\id_T, -\id_{T\dual})^* \sP\inv$ both have the defining properties of the Poincar\'e bundle.
By the uniqueness in~\labelcref{poincare}, we obtain the desired isomorphisms in case $\phi = \id_T$.
These isomorphisms will only be unique up to a constant.
But as $(0, 0)$ is a fixed point of both $(-\id_T, \id_{T\dual})$ and $(\id_T, -\id_{T\dual})$, there will be only one isomorphism of each kind respecting a fixed trivialization $\sP \big|_{(0,0)} \isom \C$.

For the general case, note that pulling back by the map $(\phi, \id_{T\dual})$ commutes with both $(-\id_T, \id_{T\dual})$ and $(\id_T, -\id_{T\dual})$, as $\phi$ is an endomorphism.
\end{proof}

Let $T$ be a complex torus of dimension $n \ge 2$ and equipped with an endomorphism $\phi$.
We consider the rank $2$ vector bundle $\sE_\phi \defn \sP_{\phi} \oplus \sP_{\phi}^{-1}$ and the $\P1$-bundle $p_\phi \from \bP(\sE_{\phi}) \to T \x T\dual$.
By \cref{poincare pullback}, the automorphisms $(-\id_T, \id_{T\dual})$ and $(\id_T, -\id_{T\dual})$ of $T \x T\dual$ induce automorphisms $i_\phi$ and $\hat i_\phi$ of $\bP(\sE_{\phi})$.
These automorphisms generate a finite group isomorphic to $\factor{\Z}{2\Z} \x \factor{\Z}{2\Z}$.
We consider the quotient
\begin{equation} \label{Qphi}
\sQ_\phi \defn \factor{\bP(\sE_\phi)}{\langle i_\phi, \hat i_\phi \rangle}.
\end{equation}
Using this notation, we can finally outline Voisin's example.

\begin{cons}[Voisin's example] \label{construction}
Let $(T, \phi)$ be a \flaccid torus of dimension $n \ge 4$.
We do the construction in the above paragraph for the given endomorphism $\phi$ and also for the endomorphism $\id_T$.
In the second case, for the sake of readability, we drop all the lower indices referring to $\id_T$.

The automorphisms $i$, $\hat i$, $i_\phi$ and $\hat i_\phi$ induce automorphisms $(i, i_\phi)$ and $(\hat i, \hat i_\phi)$ of the fibre product
\[ Z \defn \bP(\sE) \x_{T\x T\dual} \bP(\sE_{\phi}). \]
These automorphisms generate a finite group $G$, which is isomorphic to ${\factor{\Z}{2\Z} \x \factor{\Z}{2\Z}}$.
We denote the quotient by $X \defn \factor Z G$.
We get the following two commutative diagrams, where the second one is the quotient of the first one by the action of $G$:

\begin{center}
		\begin{tikzpicture}[scale=2]
		\node (A) at (0,1){$Z$};
		\node (B) at (1.5,1){$\bP(\sE_\phi)$};
		\node (C) at (0,0){$\bP(\sE)$};
		\node (D) at (1.5,0){$T\x T\dual$};
		
		\path[->,font=\scriptsize]
		(A) edge node[above]{$q_\phi$} (B)
		(A) edge node[left]{$q$} (C)
		(C) edge node[above]{$p$} (D)
		(B) edge node[right]{$p_\phi$} (D)
		(A) edge node[above]{$\pi$} (D);
		\end{tikzpicture}
                \hspace{5em}
		\begin{tikzpicture}[scale=2]
		\node (A) at (0,1){$X$};
		\node (B) at (1.5,1){$\sQ_\phi$};
		\node (C) at (0,0){$\sQ$};
		\node (D) at (1.5,0){$K(T)\x K(T\dual)$};
		
		\path[->,font=\scriptsize]
		(A) edge node[above]{} (B)
		(A) edge node[left]{} (C)
		(C) edge node[above]{} (D)
		(B) edge node[right]{} (D)
		(A) edge node[above]{} (D);
		\end{tikzpicture}
\end{center}
Here, $\sQ$ and $\sQ_\phi$ are as defined in~\labelcref{Qphi}.
\end{cons}

The interest in this construction stems from the following result of Voisin.

\begin{thm}[\protect{\cite[Theorem~4]{Voi06}}] \label{voi}
Let $X'$ be any compact complex manifold bimeromorphically equivalent to $X$.
Then $X'$ does not have the homotopy type of a complex projective manifold.
In particular, it does not admit an algebraic approximation. \qed
\end{thm}

\subsection{Local description of the singularities}

The aim of this subsection is to prove that $X$ has rigid singularities.
To this end, we examine the singularities arising in the above construction more closely.

\begin{lem}[Singularities of $\sQ$] \label{p1bundle}
The spaces $\sQ$ and $\sQ_\phi$ are $(2n + 1)$-dimensional, with only terminal quotient singularities of codimension $n+1$.
Locally analytically the singularities look like one of the following double points:
\begin{enumerate}
\item \label{p1bundle.1} $(\C^{n+1} / \pm) \x \C^n$, or
\item \label{p1bundle.2} $(\C^{2n} / \pm) \x \C$, or
\item \label{p1bundle.3} $\factor { (\C^n \x \C^n \x \C) }{ \big\langle (-\id, \id, -\id), (\id, -\id, -\id) \big\rangle }$.
\end{enumerate}
\end{lem}

\begin{proof}
The variety $\sQ_\phi$ is smooth except possibly for the image of points $x \in \bP(\sE_\phi)$ with non-trivial stabilizer.
Let $x$ be such a point and denote the fibre containing it by $F \defn p_\phi\inv \big( p_\phi(x) \big)$.
Then $p_\phi(x) \in T \x T\dual$ is a fixed point of $(-\id_T, \id_{T\dual})$, $(\id_T, -\id_{T\dual})$ or $(-\id_T, -\id_{T\dual})$.
This means $p_\phi(x) \in \tau_2(T) \x T\dual \cup T \x \tau_2(T\dual)$.

Let $\psi_1 \from U \x \C \to \sP_\phi$ be a trivialization of $\sP_{\phi}$ near $F$, where we may assume $U \subset T \x T\dual$ to be a symmetric neighbourhood of $p_{\phi}(x)$.
Consider the map $(\id_T, -\id_{T\dual}) \from U \to U$.
By \cref{poincare pullback}, there is an isomorphism $\sP_\phi\big|_U \x_U U \isom \sP_\phi\inv\big|_U$.
Using this, we obtain trivializations
\[ \begin{array}{rlcl}
\psi_2 \from & U \x \C & \to & \sP_\phi\inv\big|_U = \sP_\phi\big|_U \x_U U, \\[1ex]
& (u, t) & \mapsto & \big( \psi_1 \big( ( \id_T, -\id_{T\dual} )(u), \ t \big), \ u \big), \quad \text{and} \\[1ex]
\psi \from & U \x \C^2 & \to & \sE_\phi\big|_U, \\[1ex]
& \big( u, \ (a, b) \big) & \mapsto & \big( \psi_1(u, a), \ \psi_2(u, b) \big),
\end{array}
\]
of $\sP_{\phi}\inv\big|_U$ and $\sE_\phi\big|_U$, respectively.
Projectivizing gives a trivialization
\[ \bP(\psi) \from U \x \P1 \bij \bP(\sE_\phi)\big|_U. \]
In these coordinates the automorphisms $i_\phi, \hat{i}_\phi$ and their composition are given as
\[ \begin{array}{rlcl}
i_\phi \from & \big( u, [a : b] \big) & \mapsto & \big( ( -\id_T, \id_{T\dual} )(u), \ [b : a] \big), \\[1ex]
\hat i_\phi \from & \big( u, [a : b] \big) & \mapsto & \big( ( \id_T, -\id_{T\dual} )(u), \ [b : a] \big), \qquad \text{and} \\[1ex]
i_\phi \circ \hat i_\phi \from & \big( u, [a : b] \big) & \mapsto & \big( ( -\id_T, -\id_{T\dual} )(u), \ [a : b] \big).
\end{array}
\]
Their fixed point sets are precisely
\begin{align*}
\operatorname{Fix}(i_\phi) & = \big( U \cap ( \tau_2(T) \x T\dual ) \big) \x \big\{ [\pm 1 : 1] \big\}, \\
\operatorname{Fix}(\hat i_\phi) & = \big( U \cap ( T \x \tau_2(T\dual) ) \big) \x \big\{ [\pm 1 : 1] \big\}, \quad \text{and} \\
\operatorname{Fix}(i_\phi \circ \hat i_\phi) & = \big( U \cap ( \tau_2(T) \x \tau_2(T\dual) ) \big) \x \P1.
\end{align*}
Now, if $x$ is a fixed point of exactly one of $i_\phi$, $\hat{i}_\phi$, $i_\phi \circ \hat{i}_\phi$, then the above description in coordinates shows that locally at $x$, the quotient $\sQ_\phi$ looks like~\labelcref{p1bundle.1} or~\labelcref{p1bundle.2}, respectively.
Otherwise $x$ is a common fixed point of all three automorphisms and we get the local description~\labelcref{p1bundle.3}.
All these singularities are terminal by the Reid--Tai criterion~\cite[Theorem~3.21]{Kol13}.
To be more precise, in our situation that criterion boils down to having the eigenvalue $-1$ with multiplicity $\ge 3$ in every non-identity element of $G$, and this is clearly satisfied.
\end{proof}

\begin{lem}[Singularities of $X$] \label{sing}
The space $X$ is $(2n + 2)$-dimensional, with only terminal quotient singularities of codimension $n+1$.
Locally analytically the singularities look like one of the following double points:
\begin{enumerate}
\item \label{sing.1} $(\C^{n+2} / \pm) \x \C^n$, or
\item \label{sing.2} $(\C^{2n} / \pm) \x \C^2$, or
\item \label{sing.3} $\factor { (\C^n \x \C^n \x \C^2) }{ \big\langle (-\id, \id, -\id), (\id, -\id, -\id) \big\rangle }$.
\end{enumerate}
\end{lem}

\begin{proof}
The space $Z$ is a $\P1 \x \P1$-bundle over $T \x T\dual$.
As $G$ is finite, the quotient ${X = \factor Z G}$ is also a $(2n+2)$-dimensional complex space with only quotient singularities, contained in the image of the fixed point set of the automorphisms $g \in G \setminus \{ \id \}$.
The action of these $g$ can be described in local analytic coordinates, analougously to \cref{p1bundle}.
This gives the above local analytic description of the singularities, and the Reid--Tai criterion shows again that they are terminal.
The singular locus consists of a section over $\tau_2(T) \x T\dual \cup T \x \tau_2(T\dual)$, together with the fibres over $\tau_2(T) \x \tau_2(T\dual)$.
\end{proof}

\begin{cor}[Local rigidity] \label{rigid}
$X$ has rigid singularities.
\end{cor}

\begin{proof}
According to \cref{sing}, the variety $X$ has only quotient singularities of codimension $n + 1 \ge 5$.
Such singularities are rigid by~\cite[p.~72]{Ste03}. (Actually it suffices that the codimension is $\ge 3$.)
\end{proof}

\subsection{The topology of $X$}

The result of this last subsection says that the obstruction to algebraic approximability is not contained in the fundamental group of $X$.

\begin{prp}[Fundamental group of $X$] \label{pi1X}
The space $X$ is simply connected.
\end{prp}

\begin{proof}
To begin with, a resolution of a Kummer variety is simply connected by~\cite[Thm.~1]{Spanier56}.
By~\cite[Cor.~1.1(1)]{Tak03}, also the Kummer varieties $K(T)$ and $K(T\dual)$ themselves are simply connected, so we get that $\pi_1 \big( K(T) \x K(T\dual) \big) = 1$.
Now, the natural map $X \to K(T) \x K(T\dual)$ is a fibration with general fibre $\P1 \x \P1$, which is again simply connected.
By construction, this fibration is in fact a (locally trivial) bundle over the smooth locus of $K(T) \x K(T\dual)$.
In particular, the set of points over which the fibres are everywhere non-reduced has codimension $\ge 2$.
We conclude by~\cite[Lemma~1.5.C]{Nor83} that $\pi_1(X) = 1$.
\end{proof}

\section{$X$ does not admit an algebraic approximation}

In this section, we prove~\labelcref{main.approx}:

\begin{thm} \label{approximation}
The space $X$ from \cref{construction} does not admit an algebraic approximation.
\end{thm}

We begin with an auxiliary lemma.
It says that for locally trivial deformations, the functorial resolution of the total space is a deformation of a resolution of the central fibre.
This obviously fails if local triviality is dropped (consider e.g.~a deformation $f \from \sX \to (S, 0)$ where $\sX$ is smooth but $f\inv(0)$ is not).

\begin{lem}[Resolving locally trivial deformations] \label{ltresolution}
Let $f \from \sX \to S$ be a locally trivial deformation of a compact complex space $X \isom \sX_0$ over a smooth base $S$, and ${\pi_\sX \from \cR(\sX) \to \sX}$ the functorial resolution of $\sX$, as in \cref{funct res}.
Then, after shrinking $S$ around~$0$, the composition $f \circ \pi_\sX \from \cR(\sX) \to S$ is a locally trivial deformation of its central fibre.
Furthermore, that central fibre is a resolution of $X$.
\end{lem}

\begin{proof}
As the deformation $f$ is locally trivial, for every point $x \in X$ there are open neighbourhoods $0 \in S_x \subset S$ and $x \in U_x \subset \sX$ such that $U_x$ is isomorphic to ${(U_x \cap X) \x S_x}$ over $S$.
As the fibres of $f$ are compact, after shrinking $S$ we can assume $S_x = S$ for all $x \in X$.

Let $x \in X$ and $U \defn U_x \cap X$.
The projection $U \x S \to U$ and the open embedding $U \x S \inj \sX$ are smooth.
Hence we get for the functorial resolutions
\[ \cR(U \x S) = \cR(U) \x_U (U \x S) = \cR(U) \x S \]
and that $\cR(U\x S) \inj \cR(\sX)$ is also an open embedding.
By definition it follows that $\cR(\sX) \to S$ is a locally trivial deformation.
It is also clear that the central fibre has to be smooth.
Hence it is a resolution of $X$, via the restriction of $\pi_\sX$.
\end{proof}

\begin{proof}[Proof of \cref{approximation}]
Let $f \from \sX \to S$ be an arbitrary deformation of $X \isom \sX_0$.
Pulling back the deformation to a resolution of $S$, we may assume that $S$ is smooth.
As $X$ has rigid singularities by \cref{rigid}, the deformation $f$ is locally trivial.
After shrinking $S$, the map $\cR(\sX) \to S$ is a deformation of some resolution $\wt X$ of $X$, by \cref{ltresolution}.
According to \cref{voi}, no fibre of $\cR(\sX) \to S$ can be projective.
Then the same holds for the fibres of $\sX \to S$, because the functorial resolution is a projective morphism.
Therefore $f$ is not an algebraic approximation of $X$.
\end{proof}

\section{$\sQ$ does admit an algebraic approximation}

Keeping notation from \cref{construction}, in this section we will prove a substantial part of~\labelcref{main.mori}.

\begin{thm} \label{admitsapprox}
The space $\sQ$ admits an algebraic approximation.
\end{thm}

An approximation of $\sQ$ will be constructed out of an approximation of $T$, which is well-known to exist.
To this end, we will show that the construction of $\sQ$ can be done in families.

\subsection{The Poincar\'e bundle in families}

We show in this auxiliary section that for a family $\sX$ of complex tori, the Poincar\'e bundles belonging to the fibres $\sX_s$ locally glue together to a line bundle on the total space of the induced family $(\sX_s \x \sX\dual_s)_s$.

\begin{prp}[Deformations of the Poincar\'e bundle] \label{deformpoincare}
Let $\pi \from \sX \to S$ be a deformation of a complex torus $T \isom \sX_0$.
Then:
\begin{enumerate}
\item \label{dualfamily} Each fibre $\sX_s$ is a complex torus, and there is a deformation $p \from \sY \to S$ with fibres $\sY_s = \sX_s \times \sX_s\dual$.
\item \label{relpoincare} After shrinking $S$ around $0$, there is a line bundle $\sL$ on $\sY$ whose restriction $\sL_s \defn \sL\big|_{\sY_s}$ is isomorphic to $\sP_s$, the Poincar\'e bundle on $\sY_s$, for each $s \in S$.
\end{enumerate}
\end{prp}

The proof is based on the following computational lemma.
To fix notation, let ${T = \factor V \Lambda}$ be a complex torus.
Then we have $\pi_1(T) = \Lambda$ and consequently $\HH1.T.\Z. = \Lambda\dual \defn \Hom\Lambda.\Z.$.
By~\labelcref{dual torus}, it follows that $\HH1.T\dual.\Z. = \Lambda\ddual = \Lambda$.

\begin{lem} \label{id11}
The identity map of $\Lambda$, viewed as an element of $\HH2.T \x T\dual.\Z.$ via the natural maps
\begin{equation} \label{kunneth}
\End\Lambda. = \Lambda\dual \tensor \Lambda = \HH1.T.\Z. \tensor \HH1.T\dual.\Z. \inj \HH2.T \x T\dual.\Z.,
\end{equation}
is equal to $\cc1\sP$, the first Chern class of the Poincar\'e bundle on $T \x T\dual$.
In particular, it is of Hodge type $(1, 1)$.
\end{lem}

\begin{proof}
The inclusion map in~\labelcref{kunneth} is given by the K\"unneth formula, that is, by pulling back and taking cup product.
Furthermore, $\HH2.T \x T\dual.\Z.$ is naturally identified with the set of alternating integral $2$-forms on $\Lambda \x \Lambda\dual$.
Spelled out, this means that an element $g \tensor \mu \in \Lambda\dual \tensor \Lambda$ is sent to the following $2$-form on $\Lambda \x \Lambda\dual$:
\[ \big( (\lambda_1, f_1), (\lambda_2, f_2) \big) \mapsto g(\lambda_1) f_2(\mu) - g(\lambda_2) f_1(\mu). \]
Now, choose a basis $\gamma_1, \dots, \gamma_{2n}$ of $\Lambda$ and let $\gamma_1\dual, \dots, \gamma_{2n}\dual$ be the dual basis of $\Lambda\dual$.
Then $\id_\Lambda = \sum_{i=1}^{2n} \gamma_i\dual \tensor \gamma_i$ and by the above formula, under~\labelcref{kunneth} this gets sent to
\[ \big( (\lambda_1, f_1), (\lambda_2, f_2) \big) \mapsto f_2(\lambda_1) - f_1(\lambda_2). \]
According to~\cite[Thm.~2.5.1]{BL04}, this form represents $\cc1\sP$.
This proves the first claim.
The second one is then clear since the first Chern class of any line bundle is of type $(1, 1)$.
\end{proof}

\begin{proof}[Proof of \cref{deformpoincare}]
Any deformation of a complex torus is a complex torus, so all fibres $\sX_s$ are complex tori by~\cite[Theorem~4.1]{Cat02}. Now we consider the total space of the sheaf $\sX\dual \defn \factor{\RR1.\pi.\sO_{\sX}.}{\RR1.\pi.\Z_{\sX}.}$ on $S$.
Since $\RR1.\pi.\sO_{\sX}.$ is a vector bundle and in each fibre we are dividing out a lattice, it is clear that $\pi\dual \from \sX\dual \to S$ is a flat family of complex tori.
By definition, the fibres $(\sX\dual)_s$ are the dual tori $(\sX_s)\dual$.
Hence $\pi\dual$ is a deformation of $T\dual$, called the \emph{dual family} of $\sX$.
The fibre product $\sY \defn \sX \x_S \sX\dual$ fits into a commutative diagram
\[ \xymatrix{
\sY \ar^{r'}[rr] \ar_r[d] \ar^p[drr] & & \sX\dual \ar^{\pi\dual}[d] \\
\sX \ar^\pi[rr] & & S,
} \]
where $p \from \sY \to S$ is a deformation of $T \x T\dual$.
For each $s \in S$, the fibre $\sY_s$ is the complex torus $\sX_s \x \sX_s\dual$.
This proves~\labelcref{dualfamily}.

In order to fix the group structure on the complex tori $\sX_s$, we pick an arbitrary section $\sigma \from S \to \sX$ of $\pi$ and regard it as the zero section.
The family $\sX\dual \to S$ already comes equipped with a zero section $\tau \from S \to \sX\dual$.
Pulling back induces sections $j = ( \sigma \circ \pi\dual, \id_{\sX\dual} ) \from \sX\dual \to \sY$ of $r'$ and $i = ( \id_\sX, \tau \circ \pi ) \from \sX \to \sY$ of $r$.

After shrinking $S$, we may assume that $S$ is Stein and contractible and hence in particular the sheaf $\RR2.p.{\Z_\sY}.$ is trivial.
Consider the cohomology class $\cc1\sP$ on the central fibre $\sY_0 = T \x T\dual$.
By the triviality of $\RR2.p.{\Z_\sY}.$, this extends to a global section $\phi$ of the latter sheaf and by \cref{id11}, for all $s \in S$ the class $\phi(s) \in \HH2.\sY_s.\Z.$ continues to be the first Chern class of the Poincar\'e bundle $\sP_s$ on $\sY_s$.
In particular, $\phi(s)$ is of type $(1, 1)$ for all $s \in S$.
The pushforward of the exponential sequence on $\sY$, more precisely the exact sequence
\[ \RR1.p.{\O\sY^\x}. \lto \RR2.p.\Z_\sY. \lto \RR2.p.\O\sY., \]
then shows that $\phi$ lifts to a section $\wt\phi \in \HH0.S.{\RR1.p.\O\sY^\x.}.$, at least after shrinking $S$.
The space $S$ being Stein and contractible, the sheaf cohomology groups $\HH i.S.\sO_S.$ and $\HH i.S.\Z_S.$ vanish for $i > 0$.
By the exponential sequence on $S$, also $\HH i.S.\O S^\x.$ vanishes for $i > 0$.
Hence the five-term exact sequence associated to the Leray spectral sequence for $p$ and $\O\sY^\x$ induces an isomorphism $\Pic(\sY) \isom \HH0.S.{\RR1.p.\O\sY^\x.}.$.
This shows that the germ $\wt\phi$ comes from a line bundle $\sL$ on $\sY$.
By construction, $\sL$ has the property that $\cc1{\sL_s} = \cc1{\sP_s}$ for each $s \in S$, where $\sL_s \defn \sL\big|_{\sY_s}$.
We normalize $\sL$ by replacing it with
\[ \sL \tensor r^* \big( i^* \sL\inv \big) \tensor r'^* \big( j^* \sL\inv \big). \]
Then by the uniqueness in~\labelcref{poincare}, we have $\sL_s \isom \sP_s$ for each $s \in S$.
This is the statement of~\labelcref{relpoincare}.
\end{proof}

\subsection{Proof of \cref{admitsapprox}}

Consider the miniversal deformation of $T = \sX_0$,
\[ \pi \from \sX \to S \defn \Def(T). \]
Using notation from \cref{deformpoincare}, let $p \from \sY \to S$ be the deformation of $T \x T\dual$ with fibres $\sY_s = \sX_s \x \sX_s\dual$, and let $\sL$ be the line bundle on $\sY$ restricting to the Poincar\'e bundle on each fibre.

Consider the rank two vector bundle $\sE_S \defn \sL \oplus \sL\inv$ on $\sY$, as well as its projectivization $\bP(\sE_S) \to \sY$.
It is clear that $\bP(\sE_S) \to S$ is a deformation of $\bP(\sE)$.
Furthermore the action of $G = \factor{\Z}{2\Z} \x \factor{\Z}{2\Z}$ on the central fibre described in \cref{voisin construction} extends to all of $\bP(\sE_S)$, since the other fibres are built in the same way.
We denote by $\sQ_S$ the quotient of $\bP(\sE_S)$ by $G$.
Then $\sQ_S \to S$ is a deformation of its central fibre $(\sQ_S)_0 \isom \sQ$.

It is well-known that $\pi \from \sX \to S$ is an algebraic approximation of $T$, see~\cite[Ch.~5, Ex.~1]{Voi03}.
Also $\Picn$ of any projective variety is again projective~\cite[Prop.~7.16]{Voi02}.
Finally, projectivized vector bundles over projective varieties and finite quotients there\-of remain projective~\cite{Laz04a},~\cite[Ch.~IV, Prop.~1.5]{Knu71}.
Taken together, this shows that $\sQ_S \to S$ is an algebraic approximation of $\sQ$, as desired. \qed

\section{$X$ cannot be contracted further}

The purpose of this section is to prove~\labelcref{main.cont}.

\begin{thm}[MMP for $X$] \label{MMP X}
\label{contr}
Let $X$ be as in \cref{construction}.
\begin{enumerate}
\item\label{bimcontraction} Any bimeromorphic map $X \to X'$ onto a normal complex space $X'$ is an isomorphism.
\item\label{morifiberspace} Every run of the $K_X$-MMP immediately terminates with one of the Mori fibre spaces $X \to \sQ$ or $X \to \sQ_\phi$.
\end{enumerate}
\end{thm}

\subsection{Auxiliary results}

The following proposition is a strengthening of~\cite[Ch.~III, \S4.3,~Lemma]{Sha13} in the analytic setting.
If $\pi$ is a submersion and $Z$ is compact \kahler, then the claim follows easily from the fact that all fibres of $\pi$ have the same homology class.
However, for the applications we have in mind, $Z$ can only be assumed to be of class $\cC$ and then it may contain curves that are homologous to zero.

\begin{prp}[Maps contracting fibres of another map] \label{fibercontr}
Let $\pi \from E \to S$ be a proper surjective morphism with connected fibres between complex spaces $E$ and $S$.
Furthermore let $f \from E \to Z$ be any holomorphic map to another complex space $Z$.
\begin{enumerate}
\item \label{fibercontr.1} If for some $s_0 \in S$, the map $f$ contracts the fibre $\pi\inv(s_0)$ to a point, then it contracts all fibres $\pi\inv(s)$ for $s$ in a non-empty Zariski-open subset of $S$.
\item \label{fibercontr.2} If moreover $\pi$ is equidimensional and $S$ is locally irreducible and connected (e.g.~if $S$ is normal and irreducible), then $f$ contracts each fibre of $\pi$ to a point.
\end{enumerate}
\end{prp}

\begin{proof}
We denote the fibres of $\pi$ as $E_s \defn \pi\inv(s)$.
For~\labelcref{fibercontr.1}, we want to show that the set
\[ S_0 \defn \big\{ s\in S \;\big|\; f(E_s) \text{ is a point} \big\} \subset S \]
is Zariski-open in $S$.
We consider the graph $\Gamma$ of $f$, which is closed in $E \x Z$.
The map $\pi \x \id_Z$ is closed because $\pi$ is proper~\cite[Ch.~III, Cor.~4.3]{GPR94} and maps $\Gamma$ onto the image $\Gamma'$ of $\pi \x f$.
Thus $\Gamma'$ is an analytic subspace of $S \x Z$.
The projection $p \from \Gamma' \to S$ has fibres $p\inv(s) = \{ s \} \x f(E_s)$, and by assumption $p\inv(s_0)$ is a point.
Hence the subset of $S$ where the fibres of $p$ are zero-dimensional is non-empty, and it is Zariski-open by~\cite[Ch.~II, Thm.~1.16]{GPR94}.
This set equals $S_0$ because the fibres $E_s$ are connected.

For~\labelcref{fibercontr.2}, we assume additionally that $S$ is locally irreducible and connected.
Then equidimensionality of $\pi$ is equivalent to $\pi$ being an open map~\cite[Ch.~II, Thm.~1.18]{GPR94}.
We will show that $S_0$ is also closed in $S$.
Then connectedness of $S$ implies $S = S_0$.

If $S \setminus S_0 \ne \emptyset$, let $s \in S \setminus S_0$ be arbitrary.
Then $f(E_s)$ contains at least two distinct points $x, y \in Z$.
As $Z$ is Hausdorff, we can separate these points by disjoint open analytic neighborhoods $U_x, U_y \subset S$.
The preimages $f\inv(U_x), f\inv(U_y) \subset E$ are disjoint and open in $E$.
As $\pi$ is an open map, the set
\[ U \defn \pi \big( f\inv(U_x) \big) \cap \pi \big( f\inv(U_y) \big) \]
is an open neighborhood of $s$ in $S$.
Note that for any $t \in U$, the set $f(E_t)$ contains at least two distinct points.
Hence $U \subset S \setminus S_0$, i.e.~$S \setminus S_0$ is open in $S$.
\end{proof}

\begin{prp}[Bimeromorphic maps contract curves] \label{exc moishezon}
Let $f \from X \to Y$ be a proper bimeromorphic morphism of normal complex spaces.
Then for every $y \in Y$, the fibre $f\inv(y)$ is Moishezon.
In particular, if $f$ is not an isomorphism then there exists a compact curve $C \subset X$ which is mapped to a point by $f$.
\end{prp}

\begin{proof}
By Hironaka's Chow Lemma~\cite[Cor.~2]{HironakaFlattening}, there exists a projective bimeromorphic morphism $g \from Y' \to Y$ which factors through $f$ via a morphism $h \from Y' \to X$.
Then $h$ is automatically a bimeromorphism and closed, hence $h$ surjects for any $y \in Y$ the fibre $g\inv(y)$ onto the fibre $f\inv(y)$.
As $g\inv(y)$ is projective, the fibre $f\inv(y)$ is Moi\-shezon.
If $f$ is not an isomorphism, then some fibre $f\inv(y_0)$ is positive-dimensional.
Being Moishezon, it must contain a curve, which is then mapped to the point $y_0$.
\end{proof}

\subsection{Proof of \cref{MMP X}}

Let $\rho \from Z \to X = \factor ZG$ be the quotient map.
Let $f \from X \to X'$ be a bimeromorphic map onto a normal complex space $X'$.
As $f \circ \rho$ is proper, we can consider the Stein factorization $f \circ \rho = \rho' \circ f_Z$, where $f_Z \from Z \to Z'$ is bimeromorphic, $\rho'$ is finite and $Z'$ is normal.
If $f$ is not an isomorphism, then by \cref{exc moishezon} it contracts a curve $C \subset X$.
Let $C_Z \subset Z$ be any curve contained in $\rho\inv(C)$.
Then $f_Z$ contracts $C_Z$ and in particular $f_Z$ is not an isomorphism.
So we have reduced~\labelcref{bimcontraction} to showing that every bimeromorphic map $g \from Z \to Z'$ with $Z'$ normal is an isomorphism.

If such $g$ is not an isomorphism, then by \cref{exc moishezon} it contracts a curve $C \subset Z$.
The image $\pi(C)$ has to be a point, as $T \x T\dual$ does not contain any curves by~\cite[Lemma~7]{Voi06}.
Hence $C$ is contained in the fibre $\pi\inv \big( \pi(C) \big)$.
This fibre is isomorphic to $\P1 \x \P1$ and the restrictions of $q$ and $q_\phi$ to it are nothing but the projections onto the first and second factor, respectively.
Any curve $C$ in $\P1 \x \P1$ is numerically equivalent to an effective linear combination of the horizontal and the vertical fibre.
Hence any morphism from $\P1 \x \P1$ contracting $C$ contracts at least a horizontal or a vertical fibre.
So $g$ contracts a fibre of $q$ or a fibre of $q_\phi$.

If $g$ contracts a fibre of the $\P1$-bundle $q \from Z \to \bP(\sE)$, then by \cref{fibercontr} every fibre of $q$ is contracted by $g$.
In particular, $g$ factors through $q$, contradicting the assumption that $g$ is bimeromorphic.
Analogously, if $g$ contracts a fibre of $q_\phi$, then it factors through $q_\phi$ and we get a similar contradiction.
This proves~\labelcref{bimcontraction}.

Concerning~\labelcref{morifiberspace}, let us first note that both $X \to \sQ$ and $X \to \sQ_\phi$ are Mori fibre spaces since $K_X$ is relatively ample and the relative Picard numbers are $\rho(X / \sQ) = \rho(X / \sQ_\phi) = 1$.
Conversely, let $\psi \from X \to W$ be the first map produced by the $K_X$-MMP.
By~\labelcref{bimcontraction}, $\psi$ can be neither a divisorial nor a small contraction.
Hence $\psi$ is a \mfs.
By an argument completely analogous to the proof of~\labelcref{bimcontraction}, we see that $\psi$ has to factor through either $X \to \sQ$ or $X \to \sQ_\phi$.
In the first case, it has to be equal to $X \to \sQ$ since otherwise $\rho(X / W)$ would be at least two.
In the second case, $\psi$ is equal to $X \to \sQ_\phi$ for the same reason.
The proof of~\labelcref{morifiberspace} is thus finished. \qed

\section{Proof of \cref{main}}

Let $n \ge 10$ be an arbitrary even integer.
Pick a \flaccid torus $(T, \phi)$ of dimension $(n - 2) / 2 \ge 4$, and do \cref{construction} for this choice of $T$.
The resulting space $X$ will be our example:
Using notation from \cref{construction}, we have $X = \factor Z G$, where $Z$ is obviously uniruled and \kahler.
Hence also $X$ is uniruled, and it is \kahler by~\cite[Ch.~IV, Cor.~1.2]{Var89}.
Now, \cref{sing} and \cref{pi1X} imply~\labelcref{main.sg}, and~\labelcref{bimcontraction} is~\labelcref{main.cont}.
By~\labelcref{morifiberspace} our variety $X$ admits the \mfs $X \to \sQ$, where the base $\sQ$ admits an algebraic approximation by \cref{admitsapprox}.
This proves~\labelcref{main.mori}, with $Y = \sQ$.
However, we showed in \cref{approximation} that $X$ itself does not admit an algebraic approximation.
This is~\labelcref{main.approx}.

\providecommand{\bysame}{\leavevmode\hbox to3em{\hrulefill}\thinspace}
\providecommand{\MR}{\relax\ifhmode\unskip\space\fi MR}
\providecommand{\MRhref}[2]{%
  \href{http://www.ams.org/mathscinet-getitem?mr=#1}{#2}
}
\providecommand{\href}[2]{#2}

\end{document}